\documentclass[reqno,11pt]{amsart}

\usepackage{xcolor}
\usepackage{graphicx}
\usepackage[hidelinks]{hyperref}

\usepackage{latexsym,array}
\usepackage{amsfonts}
\usepackage{shadow}
\usepackage{amsbsy}
\usepackage{amssymb}
\usepackage{dsfont}
\usepackage{mathtools}
\usepackage{enumitem}
\usepackage{doi}

\usepackage[notref, notcite, final]{showkeys}
\usepackage{fullpage}
\usepackage{comment}
\usepackage[latin1]{inputenc}

\usepackage{cleveref}

\numberwithin{equation}{section}

\newcommand{\id}{\mathcal{I}}

\newcommand{\dom}{\mathcal{D}}

\theoremstyle{theorem}
\newtheorem{theorem}{Theorem}[section]
\newtheorem{lemma}[theorem]{Lemma}
\newtheorem{corollary}[theorem]{Corollary}
\newtheorem{proposition}[theorem]{Proposition}

\theoremstyle{definition}
\newtheorem{remark}[theorem]{{\bf Remark}}
\newtheorem{definition}[theorem]{Definition}
\newtheorem{problem}[theorem]{Problem}

\newcommand{\hh}{\mathbb{H}}

\newcommand{\rr}{\mathbb{R}}

\newcommand{\boundOP}{\mathcal{B}}
\newcommand{\closOP}{\mathcal{K}}

\newcommand{\vx}{{{x}}}

\newcommand{\Q}{\mathcal{Q}}

\renewcommand{\Re}{\mathrm{Re}}

\newcommand{\uI}{j}

\crefname{enumi}{}{}
\crefname{enumii}{}{}

\title[]
{The noncommutative fractional Fourier law in bounded and unbounded domains}

\author[F. Colombo]{Fabrizio Colombo}
\address{(FC)
Politecnico di Milano\\Dipartimento di Matematica\\Via E. Bonardi, 9\\20133
Milano, Italy}
\email{fabrizio.colombo@polimi.it}

\author[D. Deniz González]{Denis Deniz González}
\address{(DD)
	Politecnico di Milano\\Dipartimento di Matematica\\Via E. Bonardi, 9\\20133
	Milano, Italy}
\email{denis.deniz@polimi.it}

\author[S. Pinton]{Stefano Pinton}
\address{(SP)
Politecnico di Milano\\Dipartimento di Matematica\\Via E. Bonardi, 9\\20133
Milano, Italy}
\email{stefano.pinton@polimi.it}

\begin{document}

\maketitle

\begin{abstract}
Using the spectral theory on the $S$-spectrum it is possible to define the fractional powers of a large class of vector operators.
This possibility leads to new fractional diffusion and evolution problems that are of particular
interest for nonhomogeneous materials where the Fourier law is not simply
the negative gradient operator but it is
a nonconstant  coefficients differential operator of the form
$$
T=\sum_{\ell=1}^3e_\ell a_\ell(x)\partial_{x_\ell}, \ \ \ x=(x_1,x_2,x_3)\in \overline{\Omega},
$$
where, $\Omega$ can be either a bounded or an unbounded domain in $\mathbb{R}^3$ whose boundary $\partial\Omega$ is considered suitably regular,
$\overline{\Omega}$ is the closure of $\Omega$
 and $e_\ell$, for $\ell=1,2,3$ are the imaginary units of the quaternions $\mathbb{H}$.
 The operators $T_\ell:=a_\ell(x)\partial_{x_\ell}$, for $\ell=1,2,3$, are called the components of $T$ and
 $a_1$, $a_2$, $a_3: \overline{\Omega} \subset\mathbb{R}^3\to \mathbb{R}$ are the coefficients of $T$.

In this paper we study the generation of the fractional powers of $T$, denoted by $P_{\alpha}(T)$ for $\alpha\in(0,1)$,
 when the operators
$T_\ell$, for $\ell=1,2,3$ do not commute among themselves.
To define the fractional powers $P_{\alpha}(T)$ of $T$
we have to consider the weak formulation of a suitable boundary value problem
associated with the pseudo $S$-resolvent operator of $T$.
 In this paper we consider two different boundary conditions.
If $\Omega$ is unbounded we consider Dirichlet boundary conditions.
If $\Omega$ is bounded we consider the natural Robin-type boundary conditions
associated with the generation of the fractional powers of $T$ represented by the operator
$\sum_{\ell=1}^3a_\ell^2(x)n_\ell(x) \partial_{x_\ell}+a(x)I$, for $x\in \partial\Omega$,
where $I$ is the identity operator, $a:\partial\Omega \to \mathbb{R}$ is a given function and
 $n=(n_1,n_2,n_3)$ is the outward unit normal vector to $\partial\Omega$.
The  Robin-type boundary conditions associated with
 the generation of the fractional powers of $T$ are, in general, different from the
Robin boundary conditions
associated to the heat diffusion problem which leads to operators of the type
$
\sum_{\ell=1}^3a_\ell(x)n_\ell(x) \partial_{x_\ell}+b(x)I$, $x\in \partial\Omega.
$
For this reason we also discuss  the conditions  on the coefficients
 $a_1$, $a_2$, $a_3: \overline{\Omega} \subset\mathbb{R}^3\to \mathbb{R}$ of $T$
and  on the coefficient $b:\partial\Omega \to \mathbb{R}$ so that
the fractional powers  of $T$ are compatible with
 the physical Robin boundary conditions for the heat equations.

\end{abstract}
\vskip 1cm
\par\noindent
 AMS Classification: 47A10, 47A60.
\par\noindent
\noindent {\em Key words}:  Fractional  powers of vector operators,
S-spectrum, $S$-spectrum approach,  fractional diffusion processes,  Robin boundary conditions.
\vskip 1cm

\section{\bf Introduction }

Fractional diffusion and fractional evolution equations take into account nonlocal phenomena
giving a better description of the physical reality with respect to differential laws.
The most successful variation of the heat equation that takes into account nonlocal effects
is the fractional heat equation where the Laplace operator is replaced by the fractional Laplacian.
There are several ways to define fractional powers of operators which are, in general, not equivalent.
Using the spectral theory on the $S$-spectrum, see \cite{ACSBOOK,FJBOOK,CGKBOOK,MR2752913},
 a new class of fractional diffusion and evolution problems can be considered.
In particular the $S$-spectrum approach to fractional diffusion problems has been considered in \cite{FJBOOK}
where the fractional powers of quaternionic operators are systematically treated.

Using these new techniques based on the $S$-spectrum we can generate the fractional Fourier laws starting
from the differential Fourier law and the associated boundary conditions.
{\em This method has the advantage to modify only the Fourier law
without changing the conservation of energy laws in the fractional heat equation for nonhomogeneous materials}.
To recall this method and its advantages we need some notation.
An element in the quaternions $\mathbb{H}$ is of
 the form  $s=s_0+s_1e_1+s_2e_2+s_3e_3$,  where $s_0$, $s_\ell$ are real numbers ($\ell=1,2,3$)
    and ${\rm Re}(s):=s_0$ denotes the real part of $s$. The modulus
    of $s$ is defined as $|s|=(s_0^2+s_1^2+s_2^2+s_3^2)^{1/2}$
    and the conjugate is given by $\overline{s}=s_0-s_1e_1-s_2e_2-s_3e_3$.
    In the sequel we will denote by $\mathbb{S}$ the unit sphere of
    purely imaginary quaternions, an element $j$ in $\mathbb{S}$ is
    such that $j^2=-1$.

\medskip
With our approach
$\Omega$ can be a either a bounded or an unbounded domain in $\mathbb{R}^3$ whose boundary $\partial\Omega$ is sufficiently regular,
$\overline{\Omega}$ denotes the closure of $\Omega$
 and $e_\ell$, for $\ell=1,2,3$, is an orthogonal basis for the quaternions $\mathbb{H}$.
We consider vector operators
 of the form
\begin{equation}\label{TCOM}
T=\sum_{\ell=1}^3e_\ell T_\ell,
\end{equation}
where the components $T_\ell$ of $T$, $\ell=1,2,3$,  are defined by
$T_\ell:=a_\ell(x)\partial_{x_\ell}$, $x\in \overline{\Omega}$,
and we suppose that the coefficient  $a_1$, $a_2$, $a_3: \overline{\Omega} \subset\mathbb{R}^3\to \mathbb{R}$ of $T$
are not necessarily  nonconstant.
From the physical point of view the operator $T$, defined in
(\ref{TCOM}), can represent the Fourier law for nonhomogeneous materials,
 but it can represent also different physical laws.
Our goal is to generate
 the fractional powers  of $T$ when the operators
$T_\ell$, for $\ell=1,2,3$ do not commute among themselves.
The vector part of the  fractional powers $P_{\alpha}(T)$, for $\alpha\in(0,1)$, of $T$ is called the fractional
 Fourier law associated with $T$.

\medskip
It is important to observe that
using the spectral theory on the $S$-spectrum to define the fractional powers of a vector operator
$T$, one has to specify the boundary conditions associated with the operator $T$.
 When $T$ is the Fourier law for the heat diffusion problems
with the homogeneous Dirichlet boundary condition,
there are not suppletive boundary conditions that are necessary to generate the fractional powers of $T$.
In the case $\Omega$ is bounded we studied this problem in the papers  \cite{CGLax,CMPP,CPP}.
In this paper we consider the case in which $\Omega$ is unbounded.

In the paper \cite{JGPFRAC}
we have studies the fractional powers of $T$ with Robin-type boundary conditions,
 where $\Omega$ is bounded,
and the components $T_\ell$, of $T$, for $\ell=1,2,3$  commute among themselves.
It turns out that the Robin-type boundary conditions necessary to generate the fractional powers of $T$ and the
classical Robin boundary conditions of the heat equation are different.
In this paper we study
the generation of the fractional powers when the components  $T_\ell$, of $T$, for $\ell=1,2,3$ do not commute among themselves
and the relation between the two type Robin boundary conditions.

\medskip
In order to set the problem we need some results of the spectral theory on the $S$-spectrum.
 We will work in an Hilbert space but our techniques allow to define
the fractional powers of operators in quaternionic Banach spaces.

\section{\bf Problems and main results on the fractional powers of vector operators }
We consider a two-sided quaternionic Banach space $V$ and
we denote the set of closed quaternionic right linear operators on $V$ by
 $\closOP(V)$.
  The Banach space of all bounded right linear operators on $V$ is indicated by the symbol $\mathcal{B}(V)$ and is endowed with the natural operator norm.
For $T\in\closOP(V)$, we define the operator associated with the $S$-spectrum as:
\begin{equation}\label{QST}
\Q_{s}(T) := T^2 - 2\Re(s)T + |s|^2\id, \qquad \text{for $s\in\hh$}
\end{equation}
where $\Q_{s}(T):\mathcal{D}(T^2)\to V$, where $\mathcal{D}(T^2)$ is the domain of $T^2$.
 We define the $S$-resolvent set of  $T$ as
\[\rho_S(T):= \{ s\in\hh: \Q_{s}(T) \ {\rm is\ invertible\ and \ } \Q_{s}(T)^{-1}\in\boundOP(V)\}\]
and the $S$-spectrum of $T$ as
\[\sigma_S(T):=\hh\setminus\rho_S(T).\]
The operator $\Q_{s}(T)^{-1}$ is called the pseudo $S$-resolvent operator.
 For $s\in\rho_S(T)$, the left $S$-resolvent operator is defined as
\begin{equation}\label{SRESL}
S_L^{-1}(s,T):= \Q_s(T)^{-1}\overline{s} -T\Q_s(T)^{-1}
\end{equation}
and the right $S$-resolvent operator is given by
\begin{equation}\label{SRESR}
S_R^{-1}(s,T):=-(T-\id \overline{s})\Q_s(T)^{-1}.
\end{equation}
The fractional powers of $T$, denoted by $P_{\alpha}(T)$, are defined as follows:
for any $j\in \mathbb{S}$,  for $\alpha\in(0,1)$ and  $v\in\dom(T)$ we set
\begin{equation}\label{BALA1}
P_{\alpha}(T)v := \frac{1}{2\pi} \int_{-j\rr}   S_L^{-1}(s,T)\,ds_j\, s^{\alpha-1} T v,
\end{equation}
or
\begin{equation}\label{BALA2}
P_{\alpha}(T)v := \frac{1}{2\pi} \int_{-j\rr} s^{\alpha-1} \,ds_j\,  S_R^{-1}(s,T) T v,
\end{equation}
where $ds_j=ds/j$.
These formulas are a consequence
of the quaternionic version of the $H^\infty$-functional calculus based on the $S$-spectrum, see the book \cite{FJBOOK} for more details.
For the generation of the fractional powers $P_{\alpha}(T)$ a crucial assumption
on the $S$-resolvent operators is that, for $s\in \mathbb{H}\setminus \{0\}$ with ${\rm Re}(s)=0$, the estimates
\begin{equation}\label{SREST}
\left\|S_L^{-1}(s,T)\right\|_{\mathcal{B}(V)} \leq \frac{\Theta}{|s|}\quad\text{and}\quad
\left\|S_R^{-1}(s,T)\right\|_{\mathcal{B}(V)} \leq \frac{\Theta}{|s|},
\end{equation}
hold
with a constant $\Theta >0$ that does not depend on the quaternion $s$.
 It is important to observe that
the conditions (\ref{SREST}) assure that the integrals (\ref{BALA1}) and (\ref{BALA2}) are convergent and so the fractional powers are well defined.

For the definition of the fractional powers of the operator $T$ we can use equivalently the integral representation in (\ref{BALA1}) or the one in (\ref{BALA2}).
Moreover, they correspond to a modified version of Balakrishnan's formula that takes only spectral points with positive real part into account.

\medskip
A crucial problem is to determine the conditions on the coefficients
$a_1$, $a_2$, $a_3:\overline{\Omega} \subset\mathbb{R}^3\to \mathbb{R}$, of the operator $T$ defined in
(\ref{TCOM}), such that the purely imaginary quaternions are in the $S$-resolvent set $\rho_S(T)$.
This is a necessary condition, see formulas (\ref{BALA1}) and (\ref{BALA2}), since
in the quaternionic case the map
$s\mapsto s^\alpha$, for $\alpha\in (0,1)$ is not defined for $s\in(-\infty,0)$ and, unlike in the complex setting, it is not possible to choose different branches of $s^{\alpha}$ in order to avoid this problem.
For this reason it is of great importance to assume the condition $\Re(s) \geq 0$  that avoids the
half real line $(-\infty,0]$.

\medskip
Regarding the boundary conditions of Robin-type, we
will study the following problem associated with the fractional powers of the operator $T$.
\begin{problem}[Existence of the fractional powers with Robin-like boundary conditions] \label{ProbEXISTENCE}
Let $\Omega$ be a bounded domain. Let $T$ be the vector operators defined in (\ref{TCOM})
where the coefficients
$a_1$, $a_2$, $a_3: \overline{\Omega} \subset\mathbb{R}^3\to \mathbb{R}$ are suitable regular functions.
Let $F:\Omega\to \mathbb{H}$  be a given function  and
denote by  $u:\Omega\to \mathbb{H}$ the unknown function satisfying the boundary value problem:
\begin{equation}\label{ProbAA}
\begin{cases}
 &\big( T^2  -2s_0T+ |s|^2\id\big)u(x)=F(x),\ \ \ x\in \Omega,
 \\
 &
\sum_{\ell=1}^3a_\ell^2(x)n_\ell(x) \partial_{x_\ell} u(x)+a(x) u(x)=0,\ \ x\in \partial\Omega,
\end{cases}
\end{equation}
where $a:\partial\Omega \to \mathbb{R}$ is a given function and
 $n=(n_1,n_2,n_3)$ is the outward unit normal vector to $\partial\Omega$.
 \begin{itemize}
 \item[(I)]
Determine the conditions on the coefficients $a:\partial\Omega \to \mathbb{R} $, $a_1$, $a_2$, $a_3:\overline{\Omega} \subset\mathbb{R}^3\to \mathbb{R}$ such that the boundary value problem (\ref{ProbAA}) has a unique solution in a suitable function space when $\Re(s) = 0$.
\item[(II)]
Under the conditions in (I)
determine the conditions on the coefficients
such that the $S$-resolvent operators satisfy the estimates (\ref{SREST}).
\item[(III)] Consider the stationary  heat equation  for nonhomogeneous materials with Robin boundary conditions,
for $v:\Omega\to \mathbb{R}$, is given by
\begin{equation}\label{phcalr}
\begin{cases}
&
{\rm div}\, T(x) v(x) = 0,\ \ \ \ x\in \Omega,
\\
&
b(x) v(x)+ \sum_{\ell=1}^3a_\ell(x)n_\ell(x) \partial_{x_\ell}v(x)=0,\ \ x\in \partial \Omega,
\end{cases}
\end{equation}
where
$n=(n_1,n_2,n_3)$
is the outward unit normal vector to $\partial\Omega$,
and $b:\partial\Omega \to \mathbb{R}$ is a given continuous function.
Determine the conditions on the coefficients $a,\, b:\partial\Omega \to \mathbb{R} $, $a_1$, $a_2$, $a_3:\overline{\Omega} \subset\mathbb{R}^3\to \mathbb{R}$  such that the boundary condition in (\ref{phcalr}) implies the boundary condition in (\ref{ProbAA}) (see Remark \ref{r1}).
\end{itemize}
\end{problem}

\begin{remark}\label{r1}
{\rm
The operator
$$
\sum_{\ell=1}^3a_\ell^2(x)n_\ell(x) \partial_{x_\ell}
$$
is associated with the boundary condition of problem \eqref{ProbAA} that naturally arise in the definition of the bilinear form associated with the existence of the pseudo $S$-resolvent operator as a bounded linear operator,  while the operator
$$n\cdot T(x)=\sum_{\ell=1}^3a_\ell(x)n_\ell(x) \partial_{x_\ell}$$
in associated with the boundary condition of the problem \eqref{phcalr} that naturally arises as a physical flux condition.
}
\end{remark}

\begin{remark}
{\rm  In the paper \cite{JGPFRAC} we have investigated some possible solutions
of the boundary value problem (\ref{ProbAA}), in different Hilbert spaces, depending on the spectral parameter $s\in \mathbb{H}$ where the operator
$T=\sum_{\ell=1}^3e_\ell T_\ell$, defined in (\ref{TCOM}),
has commuting components $T_\ell$, for  $\ell=1,2,3$. Such analysis can be done also when
the components $T_\ell$, for $\ell=1,2,3$ do not commute.
In this paper we focus our attention on the spectral problem where $s\in \mathbb{H}$ and ${\rm Re}(s)=0$
because this is the case of interest for the definitions (\ref{BALA1}) with (\ref{BALA2}) so that we can generate
the fractional powers of $T$.
}
\end{remark}

 Regarding the boundary Dirichlet conditions for the unbounded domains, we will study
the following problem associated with the fractional powers of the operator $T$.

\medskip
\begin{problem}[Existence of the fractional powers with Dirichlet boundary conditions for unbounded domains] \label{ProbEXISTENCEbis}
Let $\Omega$ be an unbounded domain. Let $T$ be the vector operator defined in (\ref{TCOM})
where the coefficients
$a_1$, $a_2$, $a_3: \overline{\Omega} \subset\mathbb{R}^3\to \mathbb{R}$ are suitable regular functions.
Let $F:\Omega\to \mathbb{H}$  be a given function  and
denote by  $u:\Omega\to \mathbb{H}$ the unknown function satisfying the boundary value problem:
\begin{equation}\label{ProbAAbis}
\begin{cases}
 &\big( T^2  -2s_0T+ |s|^2\id\big)u(x)=F(x),\ \ \ x\in \Omega,
 \\
 &
u(x)=0,\ \ x\in \partial\Omega.
\end{cases}
\end{equation}
 \begin{itemize}
 \item[(I)]
Determine the conditions on the coefficients $a_1$, $a_2$, $a_3:\overline{\Omega} \subset\mathbb{R}^3\to \mathbb{R}$
such that the boundary value problem (\ref{ProbAAbis}) has a unique solution in a suitable function space when $\Re(s) = 0$.
\item[(II)]
Under the conditions in (I)
determine the conditions on the coefficients
such that the $S$-resolvent operators satisfy the estimates (\ref{SREST}).
\end{itemize}
\end{problem}

\subsection{\bf Summary of the main results of the paper}
In Section \ref{SSECdue} we give the weak formulation of  Problems \ref{ProbEXISTENCE} and \ref{ProbEXISTENCEbis}.
  In Section \ref{s3} we prove, under the condition $a\in \mathcal{C}^0(\partial\Omega, \mathbb{R}) $ and
  on the  coefficients $a_1$, $a_2$, $a_3\in \mathcal{C}^1(\overline{\Omega}, \mathbb{R})$ of the operator $T$
  defined in (\ref{TCOM}), the existence and the uniqueness of the weak solutions of the
  problems and suitable estimates on the  pseudo $S$-resolvent operators. Precisely we summarize the results in the following points.

\medskip
(A) The existence and uniqueness of the weak solution of Problem \ref{ProbEXISTENCE}
is stated in Theorem \ref{t3} where we define the constants
$$
C_T:=\min_{\ell=1,2,3}\inf_{ x\in\Omega}(a^2_\ell(x)),\quad
 C_T':=\sum_{i,\ell=1}^3\|a_\ell\partial_{x_\ell}a_i\|_\infty,\quad K_{a,\, \Omega}:= C^2_{\partial\Omega}\|a\|_\infty,
$$
where  $\| \cdot\|_\infty$ denotes the sup norm, and we assume
$$
C_T- C_T'C_P-K_{a, \Omega}\Big(1+C_P^2 \Big)>0 \ \ \ and \ \ \ \ \
C_T>0,
$$
where $C_P$ is the  Poincaré-Wirtinger constant and $ C_{\partial\Omega}$ are a given constant that depends on $\partial\Omega$.
Under the above conditions  the boundary value Problem (\ref{ProbAA})
has a unique weak solution $u\in \mathcal{H}(\Omega,\mathbb{H}) := \left\{u\in H^1(\Omega,\mathbb{H})  :  \int_\Omega u(x) dx=0\right\}$, for $s\in\hh\setminus \{0\}$ with $\Re(s)=0$.

\medskip
(B) In the case we work on unbounded domains the weak solution to Problems \ref{ProbEXISTENCEbis} is stated in Theorem \ref{t3bis}, i.e.,
 the boundary value Problem (\ref{ProbAAbis})
has a unique weak solution $u\in H^1_0(\Omega,\hh)$, for $s\in\hh\setminus \{0\}$ with $\Re(s)=0$ when we assume
$$
M:=\sum_{i,j=1}^3\|a_i\partial_{x_i}(a_j)\|_{L^3(\Omega)}< +\infty,
\ \ \   C_T-4M>0, \ \ \ \
 C_T>0.
$$
Observe that  the condition $M< +\infty$, in the case of unbounded domain, is a consequence of the
 the Sobolev-Gagliardo-Nirenberg inequality.

\medskip
(C) In both cases (A) and (B) we proved  the following estimates
\[
 \|u\|^2_{L^2}\leq \frac 1{s^2} \Re(b_s(u,u)),  \ \ \ \
 \|T(u)\|_{L^2}^2\leq c\,\Re(b_s(u,u)),
\]
where $c>0$ is a given constant, $b_s(u,u)$ is the bilinear form associated with the weak formulation of the problems
 and  the estimates
 hold for all $s\in\hh\setminus \{0\}$ with $\Re(s)=0$ .

\medskip
(D) In Section \ref{PRB2}, based on the estimates in point (C),  we prove
the estimates (\ref{SREST}) for the $\mathcal S$-resolvent operators and we define the fractional powers of $T$ using formula
(\ref{BALA1}) or equivalently using (\ref{BALA2}).

\medskip
(E)  Finally, consider the point {\rm (III)} of the Problem \ref{ProbEXISTENCE}. Suppose that there exists a constant $\mu$ such that the functions $a_1$, $a_2$, $a_3$  satisfy the conditions
 \begin{equation}\label{VBN}
 a_1(x)=a_2(x)=a_3(x)=\mu \ \ {\rm for\ all}\ \ x\in \partial\Omega
 \end{equation}
  and the coefficients $a$ and $b$ are such that
  \begin{equation}\label{VBNNN}
 a(x)=\mu b(x) \ \ {\rm for\ all}\ \ x\in \partial\Omega.
 \end{equation}
  Then the relation  $\sum_{\ell=1}^3a_\ell(x) n_\ell(x) \partial_{x_\ell}+b(x)I=0$,
implies $\sum_{\ell=1}^3a^2_\ell(x)n_\ell(x) \partial_{x_\ell}+a(x)I=0$, for $x\in \partial\Omega.$
Observe that, using (\ref{VBN}) and (\ref{VBNNN}), for $x\in \partial\Omega$, we have
\begin{equation}\label{BOP}
\begin{split}
\sum_{\ell=1}^3a^2_\ell(x)n_\ell(x) \partial_{x_\ell}+a(x)I
&
=\mu^2\sum_{\ell=1}^3n_\ell(x) \partial_{x_\ell}+\mu b(x)I
=\mu \Big(\sum_{\ell=1}^3a_\ell(x) n_\ell(x) \partial_{x_\ell}+b(x)I\Big).
\end{split}
\end{equation}

\section{\bf The weak formulation of the Problems \ref{ProbEXISTENCE} and \ref{ProbEXISTENCEbis}}\label{SSECdue}

In the following
$\Omega$ can be either a bounded or an unbounded domain of $\mathbb R^3$ according to the problem that we will consider.
 The boundary $\partial\Omega$ of $\Omega$ is assumed to be of class $\mathcal C^1$ even though
for some lemmas in the sequel the conditions on the open set $\Omega$ can be weakened.
We define
\[
L^p:= L^p(\Omega,\hh) := \left\{u: \Omega\to\hh: \int_{\Omega}|u(\vx)|^p\,d\vx < + \infty\right\}.
\]
The space $L^2$ with  the scalar product:
$$
\langle u,v\rangle_{L^2} := \langle u,v\rangle_{L^2(\Omega,\hh)} := \int_{\Omega} \overline{u(\vx)}v(\vx)\,d\vx,
$$
 where $u(x)=u_0(x)+u_1(x)e_1+u_2(x)e_2+u_3(x)e_3$  and $v(x)=v_0(x)+v_1(x)e_1+v_2(x)e_2+v_3(x)e_3$
for $\vx = (x_1,x_2,x_3)\in \Omega$ is a Hilbert space.
 We furthermore introduce the quaternionic Sobolev space
 \[
 H^1:= H^1(\Omega,\hh): = \Big\{u\in L^2(\Omega,\hh) : \exists \ g_{\ell,j}(x) \in L^2(\Omega,\mathbb{R}),\  \ell = 1,2,3, \ j=0,1,2,3
 \]
 \[
{\rm  such\ \  that\ }
 \int_\Omega u_j(x)\partial_{x_\ell}\varphi(x)dx=- \int_\Omega g_{\ell,j}(x)\varphi(x)dx, \ \ \forall \varphi\in \mathcal{C}^\infty_c(\Omega,\mathbb{R})
 \Big\},
 \]
 where $\mathcal{C}^\infty_c(\Omega,\mathbb{R})$ is the set of real-valued infinitely differentiable functions with compact support on $\Omega$. If $u\in H^1$ then $\partial_{x_\ell}(u_j)=g_{\ell j}$ for $\ell = 1,2,3,$ and  j$=0,1,2,3$. With the quaternionic scalar product
 \[
 \langle u,v\rangle_{H^1} := \langle u, v \rangle_{H^1(\Omega,\hh)} := \langle u,v\rangle_{L^2} + \sum_{\ell = 1}^3 \left\langle \partial_{x_{\ell}}u,\partial_{x_{\ell}}v\right\rangle_{L^2},
 \]
 we have that $ H^1(\Omega,\hh)$ becomes a quaternionic Hilbert space and the norm is defined by
 \[
 \|u\|_{H^1} ^2:= \| u \|_{H^1(\Omega,\hh)}^2 := \| u\|^2_{L^2}
 + \|u\|_D^2,
 \]
 where we have set
 $$
 \|u\|_D^2 := \sum_{\ell = 1}^3 \left\| \partial_{x_{\ell}}u\right\|_{L^2} ^2.
 $$
 As usual the space $H^1_0(\Omega, \mathbb H)$ is the closure of the space
$C^\infty_0(\Omega,\mathbb H)$ in $H^1(\Omega,\mathbb H)$ with respect to the norm $\|\cdot \|_{H^1}$.
Now we give to the problems (\ref{ProbAA}) and (\ref{ProbAAbis}) the weak formulations in order to apply the Lax-Milgram lemma in the space $H^1(\Omega,\hh)$ and $H^1_0(\Omega, \hh)$, respectively.
From the Definition \ref{TCOM} of the operator $T$,  we have
\[
\begin{split}
\Q_{s}(T) & = T^2- 2s_0T + |s|^2\id
\\
&
 =  (- (a_1\partial_{x_1})^2 - (a_2\partial_{x_2})^2 - (a_3\partial_{x_3})^2)
 \\
 & +e_1(a_3\partial_{x_3}(a_2)\partial_{x_2}-a_2\partial_{x_2}(a_3)\partial_{x_3})+e_2(a_3\partial_{x_3}(a_1)\partial_{x_1}-a_1\partial_{x_1}(a_3)\partial_{x_3}) \\
 &+e_3(a_1\partial_{x_1}(a_2)\partial_{x_2}-a_2\partial_{x_2}(a_1)\partial_{x_1}) - 2s_0T+ |s|^2\mathcal{I},
\end{split}
\]
where
$$
{\rm Scal}(\Q_{s}(T)):= (- (a_1(x)\partial_{x_1})^2 - (a_2(x_2)\partial_{x})^2 - (a_3(x_3)\partial_{x})^2+ |s|^2)\mathcal{I}
$$
is the scalar part of $\Q_{s}(T)$
and
\[
\begin{split}
{\rm Vect}(\Q_{s}(T))&:= e_1(a_3\partial_{x_3}(a_2)\partial_{x_2}-a_2\partial_{x_2}(a_3)\partial_{x_3})+e_2(a_3\partial_{x_3}(a_1)\partial_{x_1}-a_1\partial_{x_1}(a_3)\partial_{x_3})\\
& +e_3(a_1\partial_{x_1}(a_2)\partial_{x_2}-a_2\partial_{x_2}(a_1)\partial_{x_1})  - 2s_0T
\end{split}
\]
is the vector part.
 We consider the bilinear form
\[
 \langle \Q_{s}(T)u,v\rangle_{L^2} = \int_{\Omega} \overline{\Q_{s}(T)u(\vx)} v(\vx)\,d\vx
\]
for functions $u$, $v$ in class $\mathcal{C}^2(\overline{\Omega},\mathbb{H})$. Using the definition of $\Q_{s}(T)$ we have
\begin{align*}
\langle \Q_{s}(T)u,v\rangle_{L^2} = \langle T^2u,v\rangle_{L^2}- 2s_0 \langle T u,v \rangle_{L^2}+ |s|^2\langle u,v \rangle_{L^2}.
\end{align*}
Integrating by parts we obtain
\begin{align*}
 \left\langle {\rm Scal}(\Q_{s}(T)) u,v\right\rangle_{L^2} =  &    \sum_{\ell = 1}^{3}  \int_{\Omega}   \overline{a_\ell(x)\partial_{x_\ell}(u(\vx))} \left(\partial_{x_{\ell}} a_{\ell}(x)\right)v(\vx)\,d\vx\
\\
&
+     \sum_{\ell = 1}^{3}  \int_{\Omega}   \overline{a_\ell(x)\partial_{x_\ell}(u(\vx))} a_{\ell}(x)\partial_{x_{\ell}}v(\vx)\,d\vx
\\
&
- \sum_{\ell = 1}^3\int_{\partial \Omega} n_{\ell}(\vx)a^2_{\ell}(x) \left(\partial_{x_{\ell}} \overline{u(\vx)}\right)v(\vx)\,dS(\vx) + |s|^2\langle u,v \rangle_{L^2},
\end{align*}
where $dS(\vx)$ is the infinitesimal surface area of $\partial \Omega$.
 If we use the boundary condition in (\ref{ProbAA}), i.e.,
$
\sum_{\ell=1}^3a_\ell^2(x)n_\ell(x)u(x) \partial_{x_\ell}+a(x) u(x)=0,
$
we  get
\begin{align*}
 \left\langle {\rm Scal}(\Q_{s}(T)) u,v\right\rangle_{L^2} =  &
     \frac 12\sum_{\ell = 1}^{3} \int_{\Omega}    \overline{\partial_{x_\ell}(u(\vx))} \partial_{x_\ell} \left(a^2_{\ell}(x)\right)v(\vx)\,d\vx\\
& +     \sum_{\ell = 1}^{3}\int_{\Omega}  \overline{a_\ell(x)\partial_{x_\ell}(u(\vx))}\, a_\ell(x)\partial_{x_\ell}(v(\vx))\,d\vx
\\
&
+ \int_{\partial \Omega} a(x) \overline{u(\vx)} v(\vx)\,dS(\vx) + |s|^2\langle u,v \rangle_{L^2}.
\end{align*}
Instead, if we use the boundary condition in \eqref{ProbAAbis}, we obtain
\begin{align*}
 \left\langle {\rm Scal}(\Q_{s}(T)) u,v\right\rangle_{L^2} =  &
     \frac 12\sum_{\ell = 1}^{3} \int_{\Omega}    \overline{\partial_{x_\ell}(u(\vx))} \partial_{x_\ell} \left(a^2_{\ell}(x)\right)v(\vx)\,d\vx\\
& +     \sum_{\ell = 1}^{3}\int_{\Omega}  \overline{a_\ell(x)\partial_{x_\ell}(u(\vx))}\, a_\ell(x)\partial_{x_\ell}(v(\vx))\,d\vx.
\end{align*}
Relying on the above considerations we can give the following two definitions.
\begin{definition}\label{weakform}
Let $\Omega$ be a bounded domain in $\mathbb R^3$ with the boundary $\partial\Omega$ of class $\mathcal C^1$, let
$a\in \mathcal{C}^0(\partial\Omega, \mathbb{R}) $ and  $a_1$, $a_2$, $a_3\in \mathcal{C}^1(\overline{\Omega}, \mathbb{R})$.
We define the bilinear form:
\begin{equation}\label{b1}
\begin{split}
 b_s(u,v):&=   \sum_{\ell = 1}^{3}\int_{\Omega}   \overline{a_\ell(x)\partial_{x_\ell}(u(\vx))}\, a_\ell(x)\partial_{x_\ell}(v(\vx))\,d\vx +\frac 12\sum_{\ell = 1}^{3} \int_{\Omega}    \overline{\partial_{x_\ell}(u(\vx))} \partial_{x_{\ell}} \left(a^2_{\ell}(x)\right)v(\vx)\,d\vx
\\
&
+ \langle {\rm Vect}(\Q_{s}(T)) u,v \rangle_{L^2}
+ |s|^2 \langle u,v\rangle_{L^2}
+\int_{\partial \Omega} a(x) \overline{u(\vx)} v(\vx)\,dS(\vx),
\end{split}
\end{equation}
 for all functions $u,v \in H^1(\Omega,\hh)$.
\end{definition}

\begin{definition}\label{weakformbis}
Let $\Omega$ be either a bounded or an unbounded domain in $\mathbb R^3$ with the boundary $\partial\Omega$ of class $\mathcal C^1$, let $a_1$, $a_2$, $a_3\in \mathcal{C}^1(\overline{\Omega}, \mathbb{R})$.
We define the bilinear form:
\begin{equation}\label{b1bis}
\begin{split}
 b_s(u,v):&=   \sum_{\ell = 1}^{3}\int_{\Omega}   \overline{a_\ell(x)\partial_{x_\ell}(u(\vx))}\, a_\ell(x)\partial_{x_\ell}(v(\vx))\,d\vx +\frac 12\sum_{\ell = 1}^{3} \int_{\Omega}    \overline{\partial_{x_\ell}(u(\vx))} \partial_{x_{\ell}} \left(a^2_{\ell}(x)\right)v(\vx)\,d\vx
\\
&
+ \langle {\rm Vect}(\Q_{s}(T)) u,v \rangle_{L^2}
+ |s|^2 \langle u,v\rangle_{L^2},
\end{split}
\end{equation}
 for all functions $u,v \in H_0^1(\Omega,\hh)$.
\end{definition}

\begin{definition}\label{wf}
Let $\mathfrak{H}$ be the Hilbert space $H^1(\Omega,\hh)$ or some of its closed subspaces, where $\Omega$ is either a bounded or an unbounded domain in $\mathbb{R}^3$.
We say that $u\in \mathfrak{H}$ is the weak solution of the Problem \ref{ProbAA} or of the Problem \ref{ProbAAbis} for some $s\in \mathbb{H}$ if, given $F\in L^2(\Omega,\hh)$, we have
$$
b_s(u,v)=\langle F,v\rangle_{L^2},\ \ \ \ {\rm for \ all} \ \ v\in \mathfrak{H},
$$
 where $b_s$ is the bilinear form defined in \eqref{b1} or \eqref{b1bis}.
\end{definition}

\section{\bf Weak solutions of the Problems \ref{ProbEXISTENCE} and \ref{ProbEXISTENCEbis}}\label{s3}

In this section we prove existence and uniqueness of the weak solutions of Problems
 \ref{ProbEXISTENCE} and \ref{ProbEXISTENCEbis}  (Definition \ref{wf}), using Lax-Milgram lemma.
  Moreover, we need crucial estimates on the $S$-resolvent operators in order
  to define the fractional powers of the operator $T$.

To prove existence and uniqueness of the weak solutions
 it will be sufficient to show that the bilinear forms $b_s(\cdot,\cdot)$, in Definition \ref{weakform} or Definition \ref{weakformbis},
 are continuous in $H^1(\Omega,\hh)$ and they are coercive in an appropriate closed subspace of $H^1(\Omega,\hh)$
 where the choice of these subspaces depend on the boundary conditions of the problems.

  First we prove the continuity while the coercivity will be proved in Section \ref{coer1}
  for the first problem and in Section \ref{coer2} for the second one.
   As a direct consequence of the coercivity, we will prove an $L^2$ estimate for the weak solution
   $u$ that belongs to a subspace of $H^1(\Omega,\hh)$ and also we will prove an $L^2$ estimate
    for the term $T(u)$. These $L^2$ estimates will be crucial in order to prove the boundedness of the
     pseudo $S$-resolvent operator $ \Q_{s}(T)$ and the estimates \eqref{SREST}.

\medskip
We recall that the bilinear form
$$
b_s(\cdot, \cdot): H^1(\Omega,\hh)\times H^1(\Omega,\hh)\to \mathbb{H},
$$
for some $s\in \mathbb{H}$, is continuous if there exists a positive constant $C(s)$ such that
$$
|b_s(u,v)|\leq C(s) \|u\|_{H^1}\|v\|_{H^1},\ \ \ \ {\rm for \ all} \ \ u,v\in H^1(\Omega,\hh).
$$
 We note that the constant $C(s)$ depends on $s\in \mathbb{H}$ but does not depend on $u$ and $v\in H^1(\Omega,\hh)$.

 The continuity of the bilinear forms $b_s(u,v)$ can be obtained in a similar way as described in \cite{JGPFRAC} and in \cite{CMPP}. For the bilinear form \eqref{b1},  we need suitable estimates of the boundary term
\begin{lemma}\label{trace2}
Let $u\in H^1(\Omega, \mathbb H)$ and let $\Omega$ be a bounded domain with
$\partial \Omega$ is of class $\mathcal{C}^1$. Furthermore let $a\in \mathcal{C}^0(\partial\Omega,\mathbb{R})$, then we have
$$
\left| \int_{\partial\Omega} a(x)|u(x)|^2 dS(x)\right|\leq \sup_{x\in\partial\Omega} |a(x)| C^2_{\partial\Omega}\|u\|^2_{ H^1(\Omega, \mathbb H)},
$$
where $ C_{\partial\Omega}$ is the constant in formula (\ref{trace}).
\end{lemma}
\begin{proof}
It follows from the scalar valued case see \cite[p.315]{BREZIS}, precisely, suppose
 that $u\in H^1(\Omega,\mathbb{R})$ and $\Omega$ is a bounded domain in $\mathbb R^3$ with boundary $\partial\Omega$ of class $\mathcal C^1$. Then $u|_{\partial\Omega}\in H^{1/2}(\partial\Omega)$,  and there exists a positive constant $C_{\partial\Omega}$ such that
\begin{equation}\label{trace}
\|u\|_{H^{1/2}(\partial\Omega,\mathbb{R})}\leq C_{\partial\Omega}\|u\|_{H^1(\Omega,\mathbb{R})}.
\end{equation}
From estimate (\ref{trace}) we get the statement.
\end{proof}

\begin{proposition}[Continuity of $b_s$]\label{p1}
Let $\Omega$ be a  bounded domain in $\mathbb R^3$ with boundary $\partial\Omega$ of class $\mathcal C^1$. Assume that
$a\in \mathcal{C}^0(\partial\Omega, \mathbb{R}) $ and  $a_1$, $a_2$, $a_3\in \mathcal{C}^1(\overline{\Omega}, \mathbb{R})$.
Then the terms in the bilinear form $b_s(\cdot,\cdot)$ defined in \eqref{b1}
satisfy the estimates:
\begin{equation}\label{EQAZ1}
\begin{split}
& \left| \sum_{\ell = 1}^{3}\int_{\Omega}   \overline{a_\ell(x)\partial_{x_\ell}(u(\vx))}\, a_\ell(x)\partial_{x_\ell}(v(\vx))\,d\vx\right| \leq  \sup_{\ell=1,2, 3,\, x\in\Omega}(a^2_\ell(x))\|u\|_{D} \|v\|_{D},
\\
&
 \left|\frac 12\sum_{\ell = 1}^{3} \int_{\Omega}    \overline{\partial_{x_\ell}(u(\vx))} \partial_{x_{\ell}} \left(a^2_{\ell}(x)\right)v(\vx)\,d\vx\right|\leq \frac 12\sup_{\ell=1,2,3,\, x\in \Omega}(\partial_{x_\ell}(a^2_\ell(x)))\|u\|_{D}\|v\|_{L^2}
\end{split}
\end{equation}
and
\begin{equation}\label{EQAZ2NEW}
\begin{split}
&
\left|  \langle {\rm Vect}(\Q_{s}(T)) u,v \rangle_{L^2}\right|\leq \left( 2\sup_{i\neq \ell=1,2,3,\, x\in \Omega}(|a_i(x)\partial_{x_i}a_\ell(x)|) +2|s_0|\sup_{\ell=1,2,3,\, x\in \Omega}(|a_\ell(x)|)\right)\|u\|_{D}\|v\|_{L^2},
\end{split}
\end{equation}
\begin{equation}\label{EQAZ2}
|s|^2 |\langle u,v\rangle_{L^2}|\leq |s|^2 \|u\|_{L^2}\|v\|_{L^2},
\end{equation}
 while for the boundary term in \eqref{b1} the following inequality holds:
 \begin{equation}\label{EQAZ3}
 \left|\int_{\partial \Omega} a(x) \overline{u(\vx)} v(\vx)\,dS(\vx)\right|\leq  \sup_{x\in \partial\Omega}|a(x)| C^2_{\partial\Omega}\|u\|_{ H^1}\|v\|_{H^1},
 \end{equation}
 where $ C_{\partial\Omega}$ is the constant in Theorem \ref{trace}.
Moreover, the bilinear forms $b_s(\cdot,\cdot)$ are continuous from  $H^1(\Omega, \mathbb H)\times H^1(\Omega, \mathbb H) \to \mathbb{H}$,
i.e., there exits a constant $C(s)>0$ such that
\begin{equation}\label{bilcont}
|b_s(u,v)|\leq C(s) \|u\|_{H^1(\Omega,\mathbb{H})}\|v\|_{H^1(\Omega,\mathbb{H})},
\end{equation}
for all $s\in \mathbb{H}$.
\end{proposition}
\begin{proof}
The above estimates are  proved in \cite{JGPFRAC} apart from (\ref{EQAZ2NEW}) that follows by similar arguments.
\end{proof}
\subsection{Weak solution of the Problem \ref{ProbEXISTENCE}}\label{coer1}
Because of the Robin-type boundary conditions the natural space to obtain existence and uniqueness of
the weak solution of the problem \eqref{ProbAA} is the closed subspace $\mathcal H(\Omega,\hh)$ of $H^1(\Omega, \mathbb H)$ defined by
$$
\mathcal{H}(\Omega,\mathbb{H}) := \left\{u\in H^1(\Omega,\mathbb{H}) \ : \ \ \int_\Omega u(x) dx=0\right\},
$$
with the norm
$$
\|u\|_\mathcal{H}^2 :=
\|u\|_D^2 = \sum_{\ell = 1}^3 \left\| \partial_{x_\ell}u\right\|_{L^2} ^2.
$$
We adapt to the quaternionic setting the Poincaré-Wirtinger's inequality (see for example \cite[p.275]{Evans}).

\begin{corollary}\label{pw}
Let $\Omega$ be a bounded domain in $\mathbb R^3$ with boundary $\partial\Omega$ of class $\mathcal C^1$ and
let $u\in \mathcal{H}(\Omega,\mathbb{H})$. Then we have
$$
\|u\|^2_{L^2(\Omega;\mathbb{H})}\leq C_P^2 \|u\|^2_{\mathcal H}\quad\textrm{for any $u\in \mathcal H$},
$$
where $C_P$ is the Poincaré-Wirtinger constant in (\ref{PWI}).
\end{corollary}
\begin{proof}
Under the above hypotheses on  the  bounded domain $\Omega$ in $\mathbb{R}^3$
the Poincaré-Wirtinger inequality claims that for all $u\in H^1(\Omega, \mathbb{R})$ the following inequality holds:
\begin{equation}\label{PWI}
\left\|u-|\Omega|^{-1}\int_\Omega u(x)dx\right\|_{L^2(\Omega,\mathbb{R})}\leq C_P\|\nabla u\|_{L^2(\Omega,\mathbb{R})},
\end{equation}
where $C_P$ does not depend on $u$. The quaternionic case follows from estimate (\ref{PWI}).
\end{proof}

\begin{theorem}\label{t3}
Let $\Omega$ be a bounded domain in $\mathbb R^3$ with boundary $\partial\Omega$ of class $\mathcal C^1$. Assume that
$a\in \mathcal{C}^0(\partial\Omega, \mathbb{R}) $ and
 let $T$ be the operator defined in (\ref{TCOM}) with coefficients $a_1$, $a_2$, $a_3\in \mathcal{C}^1(\overline{\Omega}, \mathbb{R})$.
 Define the following constants:
\begin{equation}\label{CONTTHE26}
C_T:=\min_{\ell=1,2,3}\inf_{ x\in\Omega}(a^2_\ell(x)),\quad
C_T':=\sum_{i,\ell=1}^3\|a_\ell\partial_{x_\ell}a_i\|_\infty,\quad K_{a,\, \Omega}:= C^2_{\partial\Omega}\|a\|_\infty,
\end{equation}
where  $\| \cdot\|_\infty$ denotes the sup norm and $ C_{\partial\Omega}$ is the constant in Theorem \ref{trace}. Moreover, assume that
\begin{equation}\label{kappaomeg}
 C_T-C_T' C_P-K_{a, \Omega}\Big(1+C_P^2 \Big)>0 \ \ \ and \ \ \ \ \
C_T>0,
\end{equation}
where $C_P$ is the constant in (\ref{PWI}).
Then:

(I) The boundary value Problem (\ref{ProbAA})
has a unique weak solution $u\in\mathcal H(\Omega,\mathbb{H})$, for $s\in\hh\setminus \{0\}$ with $\Re(s)=0$,
and
\begin{equation}\label{l2es1}
 \|u\|^2_{L^2}\leq \frac 1{s^2} \Re(b_s(u,u)).
\end{equation}

(II)
Moreover, we have the following estimate
\begin{equation}\label{l2es2}
 \|T(u)\|_{L^2}^2\leq C\Re(b_s(u,u)),
 \end{equation}
for every $u\in\mathcal H(\Omega,\mathbb{H})$, and $s\in\hh\setminus \{0\}$ with $\Re(s)=0$,
where
$$
C:=1-\frac{C_T'C_p}{C_T}-\frac{K_{a,\Omega}(1+C_P^2)}{C_T}.
$$

\end{theorem}
\begin{proof}
  Step (I). To prove the existence and uniqueness of the solution for the weak solution, it is sufficient to prove the coercivity of the bilinear form $b_s(\cdot,\cdot)$, in Definition \ref{weakform},
 in since the continuity is proved in Proposition \ref{p1}. First we write explicitly ${\rm Re}\, b_{js_1}(u,u)$,
 where we have set $s=js_1$, for $s_1\in \mathbb{R}$ and  $j\in \mathbb{S}$:
\[
\begin{split}
\Re\,  b_{js_1}  (u,u) &=
s_1^2\|u\|^2_{L^2}+ \sum_{\ell=1}^3\| a_\ell\partial_{x_\ell}u\|^2_{L^2}
\\
&
+\Re\left( \frac 12\sum_{\ell = 1}^{3} \int_{\Omega}    \overline{\partial_{x_\ell}(u(\vx))} \partial_{x_{\ell}} \left(a^2_{\ell}(x_{\ell})\right)u(\vx)\,d\vx + \langle {\rm Vect}(\Q_{js_1}(T)) u,u \rangle_{L^2} \right)
\\
&
+\int_{\partial \Omega} a |u(\vx)|^2\,dS(\vx).
\end{split}
\]
By the Cauchy-Schwartz inequality and Lemma \ref{trace2}, we have
\[
\begin{split}
\Re\, b_{js_1}  (u,u)   \geq  s_1^2\|u\|^2_{L^2}
 + C_T\sum_{\ell=1}^3 \| \partial_{x_\ell} u\|^2_{L^2}
 -C_T'\sum_{\ell=1}^3\|\partial_{x_\ell} u\|_{L^2}\|u\|_{L^2}-K_{a, \Omega}\|u\|^2_{H^1}.
\end{split}
\]
Since
$$
\|u\|^2_{H^1}\leq (1+C_P^2)\|u\|^2_\mathcal H,
$$
 we obtain
\begin{equation}\label{i9}
\Re\, b_s(u,u) \geq s_1^2\|u|^2_{L^2} + \left( C_T-C_T' C_P-K_{a, \Omega}\Big(1+C_P^2 \Big)\right)\|u\|^2_\mathcal{H}.
\end{equation}
By the hypothesis, we know that
\begin{equation}\nonumber
\mathcal{K}_\Omega:= C_T-C_T' C_P-K_{a, \Omega}\Big(1+C_P^2 \Big)
>0,
\end{equation}
thus the following estimates hold:
\begin{equation}\label{i6}
 \Re\, b_{js_1}(u,u) \geq  \mathcal{K}_\Omega\|u\|_\mathcal{H}^2
\end{equation}
and
\begin{equation}\label{e5}
  \Re\, b_{js_1}(u,u)\geq s_1^2\|u\|^2_{L^2}.
\end{equation}
In particular the inequality \eqref{e5} implies the inequality \eqref{l2es1}, while the inequality \eqref{i6} implies the coercivity of $b_{js_1}(\cdot,\cdot)$ and, by the Lax-Milgram Lemma, we have that for any $w\in L^2(\Omega, \mathbb H)$ there exists a unique
$u_w\in \mathcal H$ such that
$$ b_{js_1}(u_w,v)= \langle w, v \rangle_{L^2},\ \ \ {\rm for \ all}\ \ v\in \mathcal H \
{\rm and \ for\  all}\  s_1\in \mathbb{R}.
$$
Step (II).
What remains to prove is the inequality \eqref{l2es2}. Starting from \eqref{weakform} and applying the Cauchy-Schwartz inequality, Lemma \ref{trace2} for the boundary term, Corollary \ref{pw} for the term $\|u\|_{L^2}$ and the inequality
$$ \sum_{\ell=1}^3\|\partial_{x_\ell}u\|_{L^2}^2\leq \frac{1}{C_T}\sum_{\ell=1}^3\|a_\ell\partial_{x_\ell}u\|^2 $$
we have:
\[
\begin{split}
\Re\,  b_{js_1} & (u,u) \geq \sum_{\ell=1}^3\| a_\ell\partial_{x_\ell}u\|^2_{L^2}+\int_{\partial \Omega} a |u(\vx)|^2\,dS(\vx)\\
&+\Re\left( \sum_{\ell = 1}^{3} \int_{\Omega}    \overline{a_\ell(x)\partial_{x_\ell}(u(\vx))} \partial_{x_{\ell}} \left(a_{\ell}(x)\right)u(\vx)\,d\vx + \langle {\rm Vect}(\Q_{js_1}(T)) u,u \rangle_{L^2} \right)\\
&\geq  \sum_{\ell=1}^3\| a_\ell\partial_{x_\ell}u\|^2_{L^2}-\frac{K'_{a,\, \Omega}\left(1+C^2_P\right)}{C_T}\sum_{\ell=1}^3\| a_\ell\partial_{x_\ell}u\|^2_{L^2}-C_T'\sum_{\ell=1}^3 \|\partial_{x_\ell}u\|_{L^2}\|u\|_{L^2}\\
&\geq \sum_{\ell=1}^3\| a_\ell\partial_{x_\ell}u\|^2_{L^2}-\frac{K'_{a,\, \Omega}\left(1+C^2_P\right)}{C_T}\sum_{\ell=1}^3\| a_\ell\partial_{x_\ell}u\|^2_{L^2}-C_T'C_P\sum_{\ell=1}^3 \|\partial_{x_\ell}u\|_{L^2}^2\\
&\geq \sum_{\ell=1}^3\| a_\ell\partial_{x_\ell}u\|^2_{L^2}-\frac{K'_{a,\, \Omega}\left(1+C^2_P\right)}{C_T}\sum_{\ell=1}^3\| a_\ell\partial_{x_\ell}u\|^2_{L^2}-\frac{C_T'C_P}{C_T}\sum_{\ell=1}^3 \|a_\ell\partial_{x_\ell}u\|_{L^2}^2.
\end{split}
\]
Collecting the term $\sum_{\ell=1}^3\|a_\ell\partial_{x_\ell}u\|_{L^2}^2$, observing that, with some computations, we have:
\begin{equation}\label{important}
\sum_{\ell=1}^3\|a_\ell\partial_{x_\ell}u\|_{L^2}^2= \|Tu\|_{L^2}^2
\end{equation}
and since the condition \eqref{kappaomeg} holds, we get the desired inequality \eqref{l2es2}:
\[
\begin{split}
\Re\,  b_{js_1} (u,u) & \geq \left(1-\frac{C_T'C_P}{C_T}-\frac{K_{a,\Omega}(1+C_P^2)}{C_T}\right)\sum_{\ell=1}^3\|a_\ell\partial_{x_\ell}u\|_{L^2}^2= C \|Tu\|_{L^2}^2,
\end{split}
\]
where we have set
$$
C:=1-\frac{C_T'C_P}{C_T}-\frac{K_{a,\Omega}(1+C_P^2)}{C_T}.
$$
\end{proof}
Although the technique for proving Theorem \ref{t3} is different from the technique used in Theorem $4.1$ of \cite{CGLax}, we note that the condition \eqref{kappaomeg} differs from the condition in Theorem $4.1$ of \cite{CGLax} for some terms that arise since in this article we supposed the components of $T$ are non commutative and a Robin-type condition on the boundary of $\Omega$ instead of a Dirichlet condition.

\subsection{Weak solution of the Problem \ref{ProbEXISTENCEbis}}\label{coer2}

In Theorem \ref{t3} we prove the invertibility of the operator $\mathcal Q_s(T)$.
So we adapt the strategy explained in Theorem \ref{t3} in the case of an unbounded domain
 under the more restrictive hypothesis according to that the coefficients of the first derivatives in the operator $\Q_s(T)$  are supposed to be in $L^3(\Omega, \mathbb H)$.
We need a couple of lemmas, that are well known to adapt the Sobolev-Gagliardo-Nirenberg inequality to the quaternions.

We recall formula (5) in Theorem 8.8 p.212 in \cite{BREZIS}
and we give a sketch of the proof for the sake of completeness.
\begin{lemma}\label{l2}
 For any $u\in W^{1,1}(\rr)$, we have
 \begin{equation}\label{supernova}
 \|u\|_{L^\infty(\rr)}\leq \|u'\|_{L^1(\rr)}.
\end{equation}
\end{lemma}
\begin{proof}
 We prove the statement for $v\in C^1_0(\mathbb R)$, the general case will follow from the fact that
  $C^1_0(\mathbb R)$ is dense in $W^{1,1}(\mathbb R)$. We have:
$$ v(x)=\int_{-\infty}^xv_i'(x)\, dx $$
thus we can conclude that
\begin{equation}\label{nova}
\sup_{x\in \mathbb R}|v(x)| \leq \int_{-\infty}^{+\infty}|v'(x)|\,dx.
\end{equation}
If $v\in W^{1,1}(\mathbb R)$ then there exists a sequence $v_j\in C^1_0(\mathbb R)$ such that $v_j\overset{W^{1,1}}{\longrightarrow} v$. Inequality \eqref{nova} implies the convergence of the sequence to $v$ in $L^{\infty}(\mathbb R)$. Thus the estimate \eqref{nova} holds true for any $v\in W^{1,1}(\mathbb R)$.
\end{proof}
The following lemma can be proved for $\mathbb{R}^n$ even though we will consider the case $\mathbb{R}^3$.
It is Lemma 9.4  p.278 in \cite{BREZIS}
and we give a sketch of the proof.
\begin{lemma}\label{l3}
Let $F_i\in L^{n-1}(\rr^n,\mathbb R)$ for $i=1,\dots, n$ such that $F_i$ does not depend on $x_i$. Then
$$
\int_{\rr^n} |F_1\cdots F_n|\, dV\leq \prod_{i=1}^n \left(\int_{\rr^{n-1}} |F_i|^{n-1}\, dV_i\right)^{\frac 1{n-1}},
$$
where $dV_i:= dx_1\wedge\dots\wedge dx_{i-1}\wedge dx_{i+1}\wedge\dots\wedge dx_n$.
\end{lemma}
\begin{proof}
The proof follows by an induction argument. The case $n=2$ is a consequence of the following fact:
\[
\begin{split}
\int_{\rr^2} |F_1(x_2)\cdot F_2(x_1)| dx_1\wedge dx_2&= \int_\rr |F_1(x_2)|\int_\rr |F_2(x_1)|\, dx_1\, dx_2\\
&=\int_{\rr} |F_1(x_2)| \, dx_2 \cdot \int_{\rr}|F_2(x_1)|\, dx_1.
\end{split}
\]
Now we suppose that we have proved the statement in the case $n=k-1$ when $k>2$ is an integer. By the H\"older inequality we have
\begin{equation}\label{supersupernova2}
\int_{\rr^{n-1}} |F_1\cdots F_k| dV_1\leq \left(\int_{\rr^{n-1}} |F_1|^{n-1}\, dV_1\right)^{\frac 1{n-1}}\cdot \left(\int_{\rr^{n-1}}|F_2\cdots F_k|^{\frac{n-1}{n-2}}\, dV_1\right)^{\frac{n-2}{n-1}}.
\end{equation}
By induction we obtain
\begin{equation}\label{supersupernova1}
\begin{split}
\left(\int_{\rr^{n-1}}|F_2\cdots F_k|^{\frac{n-1}{n-2}}\, dV_1\right)^{\frac{n-2}{n-1}}&\leq \left[\prod_{j=2}^n\left(\int_{\rr^{n-2}}\left(|F_j|^{\frac{n-1}{n-2}}\right)^{n-2}\, d(V_1)_j\right)^{\frac{1}{n-2}}\right]^{\frac{n-2}{n-1}}\\
& =\prod_{j=2}^n\left[\int_{\rr^{n-2}}|F_j|^{n-1}\, d(V_1)_j\right]^{\frac 1{n-1}}.
\end{split}
\end{equation}
Integrating over $x_1$ the inequality \eqref{supersupernova2} and using the inequality \eqref{supersupernova1}, we have
\[
\begin{split}
\int_{\rr^n} |F_1\cdots F_k| dV & \leq \left(\int_{\rr^{n-1}} |F_1|^{n-1}\, dV_1\right)^{\frac 1{n-1}}\cdot \int_{\mathbb R} \prod_{j=2}^n \left[\int_{\rr^{n-2}}|F_j|^{n-1}\, d(V_1)_j\right]^{\frac 1{n-1}}\, dx_1\\
& \overset{\textrm{H\"older inequality}}{\leq} \prod_{j=1}^n\left(\int_{\rr^{n-1}} |F_j|^{n-1}\, dV_j\right)^{\frac 1{n-1}},
\end{split}
\]
which concludes the proof.
\end{proof}

So we finally have the Sobolev-Gagliardo-Nirenberg inequality for the quaternions
obtained by adapting  Theorem 9.9,  p.278 in \cite{BREZIS} and using the above lemmas.
\begin{lemma}\label{l1}
For any $u\in H^1(\rr^n,\mathbb H)$, we have $u\in L^{\frac{2n}{n-2}}(\rr^n,\mathbb H)$ and the following estimate holds true
$$
\|u\|_{L^{2n/(n-2)}(\mathbb R^n,\mathbb H)}\leq K_n \sum_{i=1}^n\|\partial_{x_i}u\|_{L^2(\mathbb R^n,\mathbb H)},
$$
where
$$
 K_n:=\frac{2n-2}{n-2}.
 $$
\end{lemma}
\begin{proof}
We can suppose $u\in C^1_0(\rr^n, \mathbb H)$.  First we observe that:
\begin{equation} \label{fd}
\begin{split}
\left| \partial_{x_i}\left( |u|^{\frac{2n-2}{n-2}}\right)\right |&=\left|\partial_{x_i}\left[\left (|u|^2\right)^{\frac {n-1}{n-2}}\right]\right|
\\
&
=\frac{2n-2}{n-2} \left (|u|^2\right)^{\frac {n-1}{n-2}-1}\left |\sum_{j=0}^3 u_i\partial_{x_i}u_j \right|
\\
&
\leq \frac{2n-2}{n-2}|\partial_{x_i} u| |u|^{\frac{n}{n-2}},
\end{split}
\end{equation}
so we have
\[
\begin{split}
& \left[ \int_{\rr^n} |u|^{\frac{2n}{n-2}}\, dV\right]^{\frac{n-2}{2n}}  \leq \left[\int_{\rr^n}\, \prod_{i=1}^n\,\sup_{x_i\in \rr}|u(y_1,\dots, y_{i-1},  x_i, y_{i+1},\dots, y_n)|^{\frac 2{n-2}} \, dV \right]^{\frac{n-2}{2n}}\\
& \overset{\textrm{Lemma \ref{l3}}}{\leq} \left[\prod_{i=1}^n \left(\int_{\rr^{n-1}} \sup_{x_i\in \rr}|u(y_1,\dots, y_{i-1},  x_i, y_{i+1},\dots, y_n)|^{\frac {2n-2}{n-2}} \, dV_i\right)^{\frac 1{n-1}} \right]^{\frac{n-2}{2n}} \\
& \overset{\textrm{Lemma \ref{l2} + \eqref{fd}}}{\leq} \left[\prod_{i=1}^n \left(  \frac{2n-2}{n-2} \int_{\rr^n} |u|^{\frac n{n-2}}|\partial_{x_i} u|\, dV \right)^{\frac 1{n-1}} \right]^{\frac{n-2}{2n}} \\
& \overset{\textrm{H\"older inequality}}{\leq} \left[\left(\frac{2n-2}{n-2} \left(\int_{\rr^n} |u|^{\frac{2n}{n-2}} \, dV\right)^{\frac{1}{2}} \left( \sum_{j=1}^n \int_{\rr^n} |\partial_{x_j} u|^2\, dV\right)^{\frac 12}\right)^{\frac n{n-1}}\right]^{\frac{n-2}{2n}}.
\end{split}
\]
The above chain of inequalities can be summarized by the following inequality
 $$
 \|u\|_{L^{\frac{2n}{n-2}}(\rr^n,\mathbb H)} \leq \left(\frac{2n-2}{n-2}\right)^{\frac{n-2}{2n-2}}\|u\|_{L^{\frac{2n}{n-2}}(\rr^n,\mathbb H)}^{\frac{n}{2n-2}}\left[ \sum_{i=1}^n\|\partial_{x_i}u\|_{L^2(\rr^n,\mathbb H)} \right]^{\frac{n-2}{2n-2}}.
 $$
Thus we conclude that
\begin{equation}\label{nova2}
\|u\|_{L^{\frac{2n}{n-2}}(\rr^n,\mathbb H)} \leq  \frac{2n-2}{n-2}\sum_{i=1}^n\|\partial_{x_i}u\|_{L^2(\rr^n,\mathbb H)}.
\end{equation}
If $u\in H^1(\mathbb R^n,\mathbb H)$ then there exists a sequence $u_j\in C^1_0(\mathbb R^n, \mathbb H)$ such that $u_j\overset{H^1}{\longrightarrow} u$. Inequality \eqref{nova2} implies the convergence of the sequence to $u$ in $L^{\frac{2n}{n-2}}(\mathbb R^n,\mathbb H)$. Thus the estimate \eqref{nova2} holds true for any $u\in H^1(\mathbb R^n,\mathbb H)$.
\end{proof}

\begin{theorem}\label{t3bis}
Let $\Omega$ be an unbounded domain in $\mathbb R^3$ with boundary $\partial\Omega$ of class $\mathcal C^1$. Let $T$ be the operator defined in (\ref{TCOM}) with coefficients $a_1$, $a_2$, $a_3\in \mathcal{C}^1(\overline{\Omega}, \mathbb{R})$. Suppose that
\begin{equation}\label{c1bis}
M:=\sum_{i,j=1}^3\|a_i\partial_{x_i}(a_j)\|_{L^3(\Omega)}< +\infty
\end{equation}
and
\begin{equation}\label{kappaomegbis}
 C_T:=\min_{\ell=1,2,3}\inf_{x\in\Omega} (a^2_\ell(x))>0, \ \ \ \
  C_T-MK_3>0
\end{equation}
where $K_3=4$ is the constant in Lemma \ref{l1} for $n=3$.
Then:

(I) The boundary value Problem (\ref{ProbAAbis})
has a unique weak solution $u\in H^1_0(\Omega,\hh)$, for $s\in\hh\setminus \{0\}$ with $\Re(s)=0$,
and
\begin{equation}\label{l2es1bis}
 \|u\|^2_{L^2}\leq \frac 1{s^2} \Re(b_s(u,u)).
\end{equation}

(II)
Moreover, we have the following estimate
\begin{equation}\label{l2es2bis}
 \|T(u)\|_{L^2}^2\leq C\Re(b_s(u,u)),
 \end{equation}
for every $u\in\mathcal H^1_0(\Omega,\hh)$, and $s\in\hh\setminus \{0\}$ with $\Re(s)=0$,
where
$$
C:=\frac{C_T-MK_3}{C_T}.
$$
\end{theorem}
\begin{proof} In order to use the Lax-Milgram Lemma to prove the existence and the uniqueness of the solution for the
weak formulation of the problem, it is sufficient to prove the coercivity of bilinear form $b_s(\cdot,\cdot)$
   in Definition \ref{b1bis}
since the continuity can be proved with similar computations as in Proposition \ref{p1}. First we write explicitly ${\rm Re}\, b_{js_1}(u,u)$,
 where we have set $s=js_1$, for $s_1\in \mathbb{R}$ and  $j\in \mathbb{S}$:
	
\begin{equation}\label{e1}
\begin{split}
\Re\,  b_{js_1} & (u,u) =
s_1^2\|u\|^2_{L^2}+ \sum_{\ell=1}^3\| a_\ell\partial_{x_\ell}u\|^2_{L^2}
\\
& \quad
+\Re\left( \frac 12\sum_{\ell = 1}^{3} \int_{\Omega}    \overline{\partial_{x_\ell}(u(\vx))} \partial_{x_\ell} \left(a^2_{\ell}(x)\right)u(\vx)\,d\vx + \langle {\rm Vect}(\Q_{js_1}(T)) u,u \rangle_{L^2} \right).
\end{split}
\end{equation}

We observe that since $u\in H^1_0(\Omega,\hh)$, we can extend $u$ by $0$ outside $\Omega$ and we still have $u\in H^1(\rr^n,\hh)$. For a general function $u\in L^2(\Omega,\hh)$, we define

$$
\tilde{u}(x):=
\begin{cases}
& u(x)\quad \textrm{if $x\in\Omega$,}\\
& 0\quad \textrm{if $x\in\Omega^c$.}
\end{cases}
$$
Thus we have
\begin{equation}\nonumber
\begin{split}
& \Re\left( \frac 12\sum_{\ell = 1}^{3} \int_{\Omega}    \overline{\partial_{x_\ell}(u(\vx))} \partial_{x_\ell} \left(a^2_{\ell}(x)\right)u(\vx)\,d\vx + \langle {\rm Vect}(\Q_{js_1}(T)) u,u \rangle_{L^2(\Omega)} \right)\\
& = \Re\left( \frac 12\sum_{\ell = 1}^{3} \int_{\rr^n}    \widetilde{\overline{\partial_{x_\ell}(u(\vx))}} \widetilde{\partial_{x_\ell} \left(a^2_{\ell}(x)\right)} \widetilde u(\vx)\,d\vx + \langle  \widetilde{{\rm{Vect}(\Q_{js_1}(T)) \widetilde u}},\widetilde u \rangle_{L^2(\rr^n)} \right).
\end{split}
\end{equation}
By the H\"older inequality, Lemma \ref{l1} (for the case $n=3$) and hypothesis \eqref{c1bis}, we have that:

\begin{equation}\nonumber
\begin{split}
& \left| \frac 12\sum_{\ell = 1}^{3} \int_{\rr^3}    \widetilde{\overline{\partial_{x_\ell}(u(\vx))}} \widetilde{\partial_{x_\ell} \left(a^2_{\ell}(x)\right)} \widetilde u(\vx)\,d\vx + \langle  \widetilde{{\rm{Vect}(\Q_{js_1}(T)) \widetilde u}}, \widetilde u \rangle_{L^2(\rr^3)} \right|
\\
 &\overset{\textrm{H\"older inequality}}{\leq} \left(\sum_{\ell=1}^3 \|\widetilde{\partial_{x_\ell}u}\|_{L^2(\rr^3)} \right)\sum_{i,j=1}^3 \left(\int_{\rr^3} \left|\widetilde{a_i\partial_{x_i}a_j} \, \widetilde u \right|^2 \, dV \right)^{\frac 12}
 \\
&\overset{\textrm{H\"older inequality+\eqref{c1bis}}}{\leq} \|u\|_{D} \sum_{i,j=1}^3\left( \left( \int_{\rr^3} \left |\widetilde{a_i\partial_{x_i}a_j} \right|^{2\cdot \frac 32}\,dV\right)^{\frac 23}\left( \int_{\rr^3}\, |\widetilde u|^{2\cdot 3 }dV\right)^{\frac 13} \right)^{\frac 12}\\
&  =M\|u\|_{D}\| \widetilde u\|_{L^{6}(\rr^3)}
\\
&
\overset{\textrm{Lemma \ref{l1}}}{\leq} K_3 M\|u\|_{D} \|u\|_{D}
\\
&
= K_3 M\|u\|_{D}^2.
\end{split}
\end{equation}
The above chain of inequalities can be summarized by the following inequality
\begin{equation}\label{e2}
\begin{split}
&\left| \frac 12\sum_{\ell = 1}^{3} \int_{\Omega}    \overline{\partial_{x_\ell}(u(\vx))} \partial_{x_{\ell}} \left(a^2_{\ell}(x_{\ell})\right)u(\vx)\,d\vx  + \langle {\rm Vect}(\Q_{js_1}(T)) u,u \rangle_{L^2(\Omega)} \right|
\\
&
\leq MK_3\|u\|_{D}^2.
\end{split}
\end{equation}
Finally using the inequality \eqref{e2} in \eqref{e1}, we obtain
\[
\begin{split}
\Re\, b_{js_1}  (u,u)  & \geq  s_1^2 \|u\|^2_{L^2}
 + \left(C_T-MK_3\right)\| u\|^2_D.
\end{split}
\]
By the hypothesis \eqref{kappaomegbis} we know that
\begin{equation}\nonumber
\mathcal{K}_\Omega:= C_T-MK_3>0
\end{equation}
thus the quadratic form $b_{js_1}(\cdot,\cdot)$ is coercive for every $s_1\in \mathbb{R}$ and the following estimates hold:
\begin{equation}\label{i6bis}
 \Re\, b_{js_1}(u,u) \geq  \min \left(\mathcal{K}_\Omega, s^2_1\right)\|u\|_{H^1}.
\end{equation}
In particular we have
\begin{equation}\label{e5bis}
\Re\, b_{js_1}(u,u)\geq s_1^2\|u\|^2_{L^2}.
\end{equation}
As a consequence the inequality \eqref{e5bis} implies the inequality \eqref{l2es1}. The inequality \eqref{i6bis} implies the coercivity of $b_{js_1}(\cdot,\cdot)$ and, by the Lax-Milgram Lemma, we have that for any $w\in L^2(\Omega, \mathbb H)$ there exists $u_w\in  H^1_0(\Omega,\hh)$,
for $s_1\in \mathbb{R}\setminus \{0\}$ and  $j\in \mathbb{S}$, such that
$$ b_{js_1}(u_w,v)= \langle w, v \rangle_{L^2},\ \ \ {\rm for \ all}\ \ v\in  H^1_0(\Omega,\hh). $$
What remains to prove is the inequality \eqref{l2es2}. Starting from \eqref{b1bis}, applying the inequality \eqref{e2} and observing that
$$ \sum_{\ell=1}^3\|\partial_{x_\ell}u\|_{L^2}^2\leq \frac{1}{C_T}\sum_{\ell=1}^3\|a_\ell\partial_{x_\ell}u\|_{L^2}^2 $$
we have:
\[
\begin{split}
\Re\,  b_{js_1}  (u,u) &\geq \sum_{\ell=1}^3\| a_\ell\partial_{x_\ell}u\|^2_{L^2}+\Re\Big( \sum_{\ell = 1}^{3} \int_{\Omega}    \overline{a_\ell(x)\partial_{x_\ell}(u(\vx))} \partial_{x_{\ell}} \left(a_{\ell}(x)\right)u(\vx)\,d\vx
 \\
 &
\ \ \ \  + \langle {\rm Vect}(\Q_{js_1}(T)) u,u \rangle_{L^2} \Big)
\\
&
\geq  \sum_{\ell=1}^3\| a_\ell\partial_{x_\ell}u\|^2_{L^2}-\frac{MK_3}{C_T}\sum_{\ell=1}^3\| a_\ell\partial_{x_\ell}u\|^2_{L^2}
\\
&
\geq \frac {C_T-MK_3}{C_T}\sum_{\ell=1}^3\| a_\ell\partial_{x_\ell}u\|^2_{L^2}
\\
&
= C\|Tu\|_{L^2}^2,
\end{split}
\]
where we have set
$$
C:=\frac{C_T-MK_3}{C_T}
$$
and this concludes the proof.
\end{proof}

\begin{remark}
 As we have mentioned in the introduction the case $\Omega$ bounded with homogeneous
Dirichlet boundary conditions has already been investigated in our previous work. Above we have treated the case when $\Omega$ is unbounded.
In the case $\Omega$ is bounded the condition (\ref{c1bis}) is not required.
\end{remark}

\section{\bf The estimates for the $\mathcal S$-resolvent operators and the fractional powers of $T$}\label{PRB2}

After we prove existence and uniqueness results for the weak solutions of the problems we discussed in the previous sessions
 we can give meaning to the boundary condition
using classical results on regularity of elliptic equations up to the boundary.
In the case of Robin boundary conditions this requires the assumptions
 that the boundary has to be more regular, in the case of second order operators the boundary has to be of class $\mathcal{C}^2$ if we want to have solutions in $H^2$.
 In fact we can speak of the normal derivative $\partial_\nu u$ of a function $u\in H^2(\Omega,\mathbb{R})$
 (more in general we can set this problem in $W^{2,p}$ for $1\leq p<\infty$),
 we set $\partial_n u:=(\nabla u)|_{\partial\Omega}\cdot n$, where $n$ is the unit normal vector to $\partial\Omega$.
 This has meaning since  $(\nabla u)|_{\partial\Omega}\in L^2(\partial\Omega)$ for $\Omega \subset \mathbb{R}^N$ bounded.
 For the regularity of the Neumann problem see p.299  in \cite{BREZIS}.
 Using the estimates in Theorem \ref{t3} for the case of the Robin-type boundary conditions
or estimate in Theorem \ref{t3bis}, for the case of the Dirichlet boundary conditions in unbounded domains,
 we can now show in both cases that the $S$-resolvent operator of $T$ decays
 fast enough along the set of purely imaginary quaternions.
\begin{theorem}
Under the hypothesis of Theorem \ref{t3} or the hypothesis of Theorem \ref{t3bis}, the operator $\mathcal Q_s(T)$
 is invertible for any
 $s=js_1$, for $s_1\in \mathbb{R}\setminus \{0\}$ and  $j\in \mathbb{S}$
  and the following estimate
\begin{equation}\label{ei2}
\|\mathcal Q_s(T)^{-1}\|_{\mathcal B(L^2)}\leq \frac {1}{s_1^2}
\end{equation}
 holds.
Moreover, the $\mathcal S$-resolvent operators satisfy the estimates
\begin{equation}\label{new1}
\|\mathcal S^{-1}_L(s, \,T)\|_{\mathcal B (L^2)}\leq \frac{\Theta}{|s|}\quad\textrm{and}\quad \|\mathcal S^{-1}_R(s, \,T)\|_{\mathcal B (L^2)}\leq \frac{\Theta}{|s|},
\end{equation}
for any  $s=js_1$, for $s_1\in \mathbb{R}\setminus \{0\}$ and  $j\in \mathbb{S}$, with a constant $\Theta$ that does not depend on $s$.
\end{theorem}
\begin{proof}
We saw in Theorem \ref{t3} (respectively, Theorem \ref{t3bis}) that for all $w\in L^2(\Omega,\hh)$ there exists $u_w\in\mathcal H$ (respectively $u_w\in  H^1_0(\Omega,\mathbb H)$), for $s_1\in \mathbb{R}\setminus \{0\}$ and  $j\in \mathbb{S}$, such that
$$
 b_{js_1}(u_w,v)= \langle w, v \rangle_{L^2},\quad {\rm for \ all}\ \ v\in\mathcal H(\Omega,\mathbb H) \ \ \ (\textrm{respectively,}\quad {\rm for \ all}\ \ v\in H^1_0(\Omega, \mathbb H)).
 $$
Thus we can define the inverse operator $\mathcal Q_{js_1}(T)^{-1}(w):=u_w$ for any $w\in L^2(\Omega, \mathbb H)$ (we note that the range of $\mathcal Q_{js_1}(T)^{-1}$ is in $\mathcal H(\Omega,\mathbb H)$ (respectively in $H^1_0(\Omega,\mathbb H)$)). The inequality \eqref{l2es1} (respectively \eqref{l2es1bis}), applied to $u:=\mathcal Q_{js_1}(T)^{-1}(w)$, implies:
\begin{equation}
\begin{split}
s_1^2\|\mathcal Q_{js_1}(T)^{-1}(w)\|^2_{L^2} & \overset{\textrm{\eqref{l2es1} (respectively \eqref{l2es1bis})}}{\leq} {\rm Re}\, b_{js_1}(Q_{js_1}(T)^{-1}(w),Q_{js_1}(T)^{-1}(w))
\\
&
\leq |b_{js_1}(Q_{js_1}(T)^{-1}(w),Q_{js_1}(T)^{-1}(w))|\\
& \leq |\langle w, Q_{js_1}(T)^{-1}(w)\rangle_{L^2}|
\\
&
\leq \|w\|_{L^2}\|Q_{js_1}(T)^{-1}(w)\|_{L^2}, \quad\textrm{for any $w\in L^2(\Omega, \mathbb H)$}.
\end{split}
\end{equation}
Thus we have
$$
 \|\mathcal Q_{js_1}(T)^{-1}\|_{\mathcal B(L^2)}\leq \frac {1}{s_1^2},
  \ \ \ {\rm for} \  s_1\in \mathbb{R}\setminus \{0\}\ \ {\rm and}  \ \ j\in \mathbb{S}.
$$
The estimates \eqref{new1} follow from the estimate \eqref{l2es2} (respectively \eqref{l2es2bis}). Indeed we have
\begin{equation}\nonumber
\begin{split}
 C\|Tu_w\|^2&\overset{\textrm{\eqref{l2es2} (respectively \eqref{l2es2bis})}}{\leq}  {\rm Re}(b_{js_1}(u_w,u_w))
 \\
 &
 \leq |b_{js_1}(u_w,u_w)|
 \\
 &
 \leq |\langle w,u_w \rangle_{L^2}|
 \\
 &\leq \|w\|_{L^2}\|u_w\|_{L^2}
 \\
 &
 \overset{\eqref{ei2}}{\leq} \frac {1}{s_1^2}\|w\|^2_{L^2},
 \end{split}
 \end{equation}
  for $s_1\in \mathbb{R}\setminus \{0\}$  and $j\in \mathbb{S}$.
This estimate implies
\[
\left\| T\Q_{js_1}(T)^{-1}w\right\|_{L^2} = \|Tu_{w}\|_{L^2} \leq \frac{1}{\sqrt{C}|s_1|} \|w\|_{L^2}
\]
thus we obtain
\begin{equation}\label{new4}
\left\| T \Q_{js_1}(T)^{-1}\right\|_{\mathcal B(L^2)} \leq \frac{1}{\sqrt{C}|s_1|}.
\end{equation}
In conclusion, if we set
\[
\Theta := 2\max\left\{1,\frac{1}{\sqrt{C}}\right\},
\]
estimates \eqref{new4} and \eqref{ei2} yield
\begin{equation}\label{eq:S_R_estimate}
\begin{split}
\left\|S_R^{-1}(s,T)\right\|_{\mathcal B(L^2)} &= \left\|(T - \overline{s}\id)\Q_{s}(T)^{-1}\right\|_{\mathcal B(L^2)} \\
&\leq \left\|T\Q_{s}(T)^{-1}\right\|_{\mathcal B(L^2)} + \left\|\overline{s}\Q_{s}(T)^{-1}\right\|_{\mathcal B(L^2)} \leq\frac{ \Theta}{|s_1|}
\end{split}
\end{equation}
and
\begin{equation*}
\begin{split}
\left\|S_L^{-1}(s,T)\right\|_{\mathcal B(L^2)} &= \left\|T\Q_{s}(T)^{-1} - \Q_{s}(T)^{-1}\overline{s}\right\|_{\mathcal B(L^2)} \\
&\leq \left\|T\Q_{s}(T)^{-1}\right\|_{\mathcal B(L^2)} + \left\| \Q_{s}(T)^{-1}\overline{s}\right\|_{\mathcal B(L^2)} \leq\frac{ \Theta}{|s_1|},
\end{split}
\end{equation*}
for any $s = \uI s_1\in\hh\setminus\{0\}$.
\end{proof}

Thanks to the above results, we are now ready to establish our main statement.
\begin{theorem}\label{thm:main}
Under the hypothesis of Theorem \ref{t3} or the hypothesis of Theorem \ref{t3bis}, for any $\alpha\in(0,1)$ and $v\in\dom(T)$, the integral
	\[
	P_{\alpha}(T)v := \frac{1}{2\pi}\int_{-\uI\rr} s^{\alpha-1}\,ds_{\uI}\,S_{R}^{-1}(s,T) Tv
	\]
	converges absolutely in $L^2$.
\end{theorem}
\begin{proof}
	The right $S$-resolvent equation implies
	\[
	S_{R}^{-1}(s,T)Tv = sS_{R}^{-1}(s,T)v - v,\qquad \forall v\in\dom(T)
	\]
	and so
	\begin{align*}
	\frac{1}{2\pi}\int_{-\uI\rr} \left\|s^{\alpha-1}\,ds_{\uI}\,S_{R}^{-1}(s,T) Tv\right\|_{L^2}
	&\leq \frac{1}{2\pi}\int_{-\infty}^{-1} |t|^{\alpha-1} \left\| S_{R}^{-1}(-\uI t,T) \right\|_{\mathcal B(L^2)} \left\|Tv\right\|_{L^2} \,dt\\
	& \quad
	+\frac{1}{2\pi}\int_{-1}^{1} |t|^{\alpha-1} \left\| (-\uI t) S_{R}^{-1}(-\uI t,T)v - v\right\|_{L^2} \,dt\\
	& \quad
	+ \frac{1}{2\pi}\int_{1}^{+\infty} t^{\alpha-1} \left\| S_{R}^{-1}(\uI t,T) \right\|_{\mathcal{B}(L^2)} \left\|Tv\right\|_{L^2}\,dt.
	\end{align*}
	As $\alpha\in(0,1)$, the estimate \eqref{new1} now yields
	\begin{align*}
	\frac{1}{2\pi}\int_{-\uI\rr} &\left\|s^{\alpha-1}\,ds_{\uI}\,S_{R}^{-1}(s,T) Tv\right\|_{L^2}
    \\
	&\leq \frac{1}{2\pi}\int_{1}^{+\infty} t^{\alpha-1}\frac{\Theta}{t}\left\|Tv\right\|_{L^2} \,dt +
	\frac{1}{2\pi}\int_{-1}^{1} |t|^{\alpha-1} \left(|t| \frac{\Theta}{|t|} + 1\right)\|v\|_{L^2} \,dt\\
	& \quad
	+\frac{1}{2\pi}\int_{1}^{+\infty} t^{\alpha-1} \frac{\Theta}{t} \left\|Tv\right\|_{L^2}\,dt \\
	&< +\infty.
	\end{align*}
\end{proof}

We conclude this paper with some comments.

(I) In the literature there are several non linear models that involve the fractional Laplacian and even the fractional powers
of more general elliptic operators, see for example, the books
\cite{BocurValdinoci,Vazquez}.

(II)
The $S$-spectrum approach to fractional diffusion problems used in this paper is a generalization of the
method developed by   Balakrishnan, see \cite{Balakrishnan}, to
define the fractional  powers of a real operator $A$.
In the paper \cite{64FRAC}
following the book of M. Haase, see
\cite{Haase}, has been developed
the theory on fractional powers of quaternionic linear operators, see also \cite{Hinfty,FJTAMS}.

(III) The spectral theorem on the $S$-spectrum is also an other tool
to define the fractional powers of vector operators, see \cite{ack} and for perturbation results see \cite{CCKSpert}.

(IV) An historical note on the discovery of the $S$-resolvent operators and of the $S$-spectrum
can be found in the introduction of the book \cite{CGKBOOK}.

The most important results in quaternionic operators theory based on the $S$-spectrum
and the associated theory of slice hyperholomorphic functions
are contained in the books \cite{COF,ACSBOOK,FJBOOK,CGKBOOK,ACSBOOK2,MR2752913,BOOKGS,GSSb}, for the case on $n$-tuples of operators see
\cite{JFACSS}.

(V) Our future research directions will consider the development of ideas from one and several
complex variables, such as in \cite{BKP,BKPZ,BPZ,BDP,HPR,MPR,MP2}  to the quaternionic setting.


\begin{thebibliography}{10}



\bibitem{ack}
D.~{Alpay}, F.~{Colombo},  D. P. Kimsey,
{\em  The spectral theorem for for quaternionic unbounded normal operators based on the $S$-spectrum},
J. Math. Phys., {\bf 57}(2016), pp. 023503, 27.





\bibitem{Hinfty}
 D. Alpay, F. Colombo, T. Qian, I. Sabadini,
{\em The $H^\infty$functional calculus based on the S-spectrum for quaternionic operators and for n-tuples of noncommuting operators},
 J. Funct. Anal., {\bf 271}(2016), pp. 1544--1584.





\bibitem{ACSBOOK}
D. Alpay, F. Colombo, I. Sabadini,
{\em Slice Hyperholomorphic Schur Analysis},
Volume 256 of {\em Operator Theory: Advances and Applications}. Birkh\"{a}user, Basel. 2017.


\bibitem{COF}
D. Alpay, F. Colombo, I. Sabadini,
{\em Quaternionic de Branges spaces and characteristic operator function},
  SpringerBriefs in Mathematics, Springer, Cham, 2020.


\bibitem{Balakrishnan}
A. V. Balakrishnan,
{\em  Fractional powers of closed operators and the semigroups generated by them},
Pacific J. Math., {\bf 10}(1960), pp. 419--437.


\bibitem{BKP}
L. Baracco, T. V. Khanh, S. Pinton,  {\em The complex Monge-Amp\'ere
equation on weakly pseudoconvex domains},
 C. R. Math. Acad. Sci. Paris, {\bf 355} (2017), no. 4, 411--414.

\bibitem{BKPZ}
L. Baracco, T. V. Khanh, S. Pinton,  G. Zampieri,
{\em  H\"older regularity of the solution to the complex Monge-Amp\'ere equation with $L^p$
density},
 Calc. Var. Partial Differential Equations, {\bf 55} (2016),
no. 4, Art. 74, 8 pp.

\bibitem{BPZ}
L. Baracco,  S. Pinton,  G. Zampieri,
{\em H\"older regularity Hypoellipticity of the Kohn--Laplacian $\Box_b$
and of the $\bar\partial$-Neumann problem by means of subelliptic
multiplier},
 Math. Ann., {\bf 362} (2015), no. 3-4, 887--901.

\bibitem{BDP}
E. Barletta, S. Dragomir, M. M. Peloso,
{\em  Worm domains and Fefferman space-time singularities},
 J. Geom. Phys., {\bf 120} (2017), 142--168.



\bibitem{BREZIS}
H. Brezis, {\em Functional analysis, Sobolev spaces and partial differential equations},
 Universitext. Springer, New York, 2011. xiv+599 pp.

\bibitem{BocurValdinoci}
 C. Bucur, E. Valdinoci,
  {\em Nonlocal diffusion and applications},
   Volume 20 of {\em Lecture Notes of the Unione Matematica Italiana},
   Springer, [Cham]; Unione Matematica Italiana, Bologna. 2016.


\bibitem{CCKSpert}
 P. Cerejeiras, F. Colombo, U. K\"ahler, I. Sabadini,
{\em Perturbation of normal quaternionic operators},
Trans. Amer. Math. Soc., {\bf 372} (2019),  3257--3281.

 \bibitem{FJTAMS}
 F. Colombo, J. Gantner,
{\em Fractional  powers of quaternionic operators and Kato's formula using slice hyperholomorphicity},
 Trans. Amer. Math. Soc., {\bf 370} (2018), no. 2, 1045--1100.

\bibitem{CGLax}
F. Colombo, J. Gantner,
{\em Fractional powers of vector operators and fractional Fourier's law in a Hilbert space}
Journal of Physics A: Mathematical and  Theoretical, {\bf 51} (2018), 305201 (25pp).


\bibitem{64FRAC}
F. Colombo, J. Gantner,
{\em An application of the $S$-functional calculus to fractional diffusion processes},
 Milan J. Math., {\bf 86} (2018), 225--303.



\bibitem{FJBOOK}
F. Colombo, J. Gantner,
{\em Quaternionic closed operators, fractional powers and fractional diffusion processes},
 Operator Theory: Advances and Applications, 274. Birkh\"auser/Springer, Cham, 2019. viii+322 pp.
 ISBN: 978--3--030--16408--9; 978--3--030--16409.

\bibitem{CGKBOOK} F. Colombo, J. Gantner, D.P. Kimsey,
{\em Spectral theory on the $S$-spectrum for quaternionic operators},
Operator Theory: Advances and Applications, 270.
Birkh\"auser/Springer, Cham, 2018. ix+356 pp. ISBN: 978-3-030-03073-5; 978-3-030-03074-2 47-02.



\bibitem{CMPP}  F. Colombo, S. Mongodi, M. Peloso, S. Pinton,
{\em Fractional powers of the non commutative Fourier's laws by the  S-spectrum approach},
Math. Methods Appl. Sci., {\bf 42} (2019), no. 5,  1662--1686.



\bibitem{CPP}
F. Colombo, M. Peloso, S. Pinton,
{\em The structure of the fractional powers of the noncommutative Fourier law},
 Math. Methods Appl. Sci., {\bf 42} (2019), no. 18, 6259--6276.


\bibitem{JGPFRAC}
F. Colombo, D. Deniz-Gonzales,  S. Pinton,
{\em
Fractional powers of vector operators with first order boundary conditions},
  J. Geom. Phys., {\bf 151} (2020), 103618.




\bibitem{ACSBOOK2}
F. Colombo, I. Sabadini, D.C. Struppa,
{\em  Entire slice regular functions},
 SpringerBriefs in Mathematics. Springer, Cham, 2016. v+118 pp. ISBN: 978-3-319-49264-3; 978-3-319-49265-0.


 \bibitem{MR2752913}
F. Colombo, I. Sabadini, D.~C. Struppa,
 {\em Noncommutative functional calculus. Theory and applications of slice hyperholomorphic functions},
 Volume 289 of {\em Progress
  in Mathematics}.
 Birkh\"auser/Springer Basel AG, Basel. 2011.


 \bibitem{JFACSS}
 F. Colombo, I. Sabadini, D.~C. Struppa,
 {\em A new functional calculus for non commuting operators},
  J. Funct. Anal., {\bf 254} (2008), 2255-2274.

 \bibitem{Evans} L.C. Evans, Partial differential equations. Second edition. Graduate Studies in Mathematics, 19.
 American Mathematical Society, Providence, RI, 2010. xxii+749 pp.


\bibitem{BOOKGS}
 S. G. Gal,  I. Sabadini,
{\em Quaternionic approximation. With application to slice regular functions},
 Frontiers in Mathematics. Birkh\"auser/Springer, Cham, 2019. x+221 pp. ISBN: 978-3-030-10664-5; 978-3-030-10666-9.


\bibitem{GSSb}
 G. Gentili, C. Stoppato, D. C.  Struppa,
  {\em Regular functions of a quaternionic variable},
Volume of {\em Springer Monographs in Mathematics}. Springer, Heidelberg. 2013.


\bibitem{Haase}
M. Haase,
  {\em The functional calculus for sectorial operators},
   Volume 169 of {\em Operator Theory: Advances and Applications}. Birkh\"{a}user, Basel. 2006.

\bibitem{HPR} P. Harrington, Marco M. Peloso, A. Raich,
 {\em Regularity equivalence of the Szeg\"o projection and the complex Green operator},
 Proc. Amer. Math. Soc.,  {\bf 143} (2015), no. 1, 353--367.


\bibitem{MP2}
A. Monguzzi, M. M. Peloso,
 {\em  Sharp estimates for the Szeg\"o projection on the
distinguished boundary of model worm domains},
 Integral Equations Operator Theory, {\bf 89} (2017), no. 3,
315-344.

\bibitem{MPR} D. H. M\"uller, M. M. Peloso, F. Ricci,
{\em Analysis of the Hodge Laplacian
on the Heisenberg group},
  Memoirs Amer. Math. Soc., {\bf 233} (2015), no. 1095.

\bibitem{Vazquez}
J. L. Vazquez,
{\em The porous medium equation. Mathematical theory.}
Volume of {\em Oxford Mathematical Monographs}.
The Clarendon Press, Oxford University Press, Oxford. 2007.







\end{thebibliography}
\end{document}